\documentclass{imsart}

\RequirePackage[OT1]{fontenc}
\RequirePackage{amsthm,amsmath}
\RequirePackage[numbers]{natbib}
\RequirePackage[colorlinks,citecolor=blue,urlcolor=blue]{hyperref}

\usepackage{amssymb}
\usepackage{amsfonts}
\usepackage{amsthm}
\usepackage{amsmath}
\usepackage{dsfont}

\usepackage{stmaryrd}

\renewcommand{\leq}{\leqslant}
\renewcommand{\geq}{\geqslant}

\newcommand{\inner}[2]{\langle #1, #2 \rangle}
\newcommand{\innerl}[2]{\left\langle #1, #2 \right\rangle}

\newtheorem{theorem}{Theorem}[section]
\newtheorem{proposition}{Proposition}[section]

\usepackage{times}
\usepackage{graphicx,epstopdf} 
\usepackage[T1]{fontenc}

\usepackage{subfigure}

\DeclareMathOperator{\cov}{cov}
\DeclareMathOperator{\MSE}{MSE}
\DeclareMathOperator{\ess}{ess}
\DeclareMathOperator{\E}{\mathbb{E}\,}

\DeclareMathOperator{\tr}{tr}

\startlocaldefs
\numberwithin{equation}{section}
\theoremstyle{plain}

\endlocaldefs

\begin{document}

\begin{frontmatter}
\title{Finite impulse response models: A non-asymptotic analysis of the least squares estimator}
\runtitle{Finite impulse response models}

\begin{aug}
\author{\fnms{Boualem } \snm{Djehiche}\thanksref{}\ead[label=e1]{boualem@kth.se}},
\author{\fnms{Othmane } \snm{Mazhar}\thanksref{}\ead[label=e2]{othmane@kth.se}}
\\
\and 
\author{\fnms{Cristian } \snm{R. Rojas}\thanksref{}
\ead[label=e3]{cristian.rojas@ee.kth.se}
\ead[label=u1,url]{http://www.foo.com}}

\runauthor{B. Djehiche, O. Mazhar and C. R. Rojas.}

\affiliation{Royal Institute of Technology\thanksmark{m1}}

\address{Boualem Djehiche\\
Department of Mathematics\\
School of Engineering Sciences\\
KTH Royal Institute of Technology\\
100 44 Stockholm, Sweden.\\
\printead{e1}\\
\phantom{E-mail:\ }}

\address{Othmane Mazhar and Cristian R. Rojas\\
Division of Decision And Control Systems\\
School of Electrical Engineering and Computer Science\\
KTH Royal Institute of Technology\\
SE-100 44 Stockholm, Sweden.\\
\printead{e2}\\
\phantom{E-mail:\ }\printead*{e3}}

\end{aug}

\begin{abstract}
We consider a finite impulse response system with centered independent sub-Gaussian design covariates and noise components that are not necessarily identically distributed. We derive non-asymptotic near-optimal estimation and prediction bounds for the least squares estimator of the parameters. Our results are based on two concentration inequalities on the norm of sums of dependent covariate vectors and on the singular values of their covariance operator that are of independent value on their own and where the dependence arises from the time shift structure of the time series. These results generalize the known bounds for the independent case.
\end{abstract}



\end{frontmatter}

\section{Introduction}

In the present article, we study general Finite Impulse Response (FIR) systems, which are in some sense the simplest discrete-time dynamical systems. More precisely, let $x = (x_t)_{t=2-T}^{\infty}$ be a sequence of independent random variables, corresponding to the input of the system, the integer $T\ge 1$ be the time lag of the response and $y = (y_t)_{t=1}^{\infty}$ the output of the system. General (nonlinear) FIR systems correspond to the following additive-noise time series:
\begin{equation*}
    y_t = f(x_t,x_{t-1},\dots,x_{t-T+1}) + \varepsilon_t,  \qquad t \in \llbracket  1,\infty \rrbracket, 
\end{equation*}
where $\llbracket  a,b \rrbracket$ denotes the set $\{n \in \mathbb{Z}\colon a \leq n \leq b \}$, $\varepsilon = (\varepsilon_t)_{t=1}^{\infty}$ is a sequence of unobserved, independent random variables that represents an additive noise to the system and $f:\mathbb{R}^T \to \mathbb{R}$ is an unknown regressor. Of special interest is the linear regressor case, where $a = [a_1,a_2,\dots,a_T]^\top \in \mathbb{R}^T$ is the parameter vector and
\begin{equation*}
    y_t = a_1 x_t + \dots + a_T x_{t-T+1} + \varepsilon_t, \qquad t \in  \llbracket  1,\infty \rrbracket. 
\end{equation*}
Upon observing the sequence $y$ up to time $N$, we obtain the linear regression model
\begin{equation}\label{eq: FIR as linear regression} 
    y_t = a^\top X_t +\varepsilon_t, \qquad  t \in  \llbracket  1,N \rrbracket 
\end{equation}
with 'time shifted' covariates
\begin{align} \label{eq:time_shifted}
X_t = [x_t \;\; x_{t-1} \;\; \cdots \;\; x_{t-T+1}]^\top.
\end{align}
We refer to the matrix $X= [X_1, X_2,\dots, X_N]^\top$ as the design matrix, to $\hat{\Sigma} := \frac{1}{N} X^\top X$ as the sample covariance matrix, to $y = [y_1,y_2,\dots,y_N]^\top $ as the output vector and to $\varepsilon = [\varepsilon_1,\varepsilon_2,\dots,\varepsilon_N]^\top $ as the noise vector, the linear FIR model in vector format becomes
\begin{equation}\label{eq: FIR as linear regression in vector format}
    y = Xa +\varepsilon.
\end{equation}
Note that the dynamical structure creates dependence between the covariates, which makes this setup different from the classical regression problem with independent covariate design.

The study of linear dynamical systems is currently an active area of research within the statistical learning community due to its potential relation to the reinforcement learning problem \cite{DBLP:journals/corr/abs-1802-03981,NIPS2017_7247,boczar2018finitedata,DBLP:journals/corr/abs-1902-01848} where the focus is mainly on providing non-asymptotic bounds on the accuracy of the estimated parameters of the linear state-space models \cite{DBLP:journals/corr/abs-1812-01251,pmlr-v75-simchowitz18a}. State-space models can be viewed as Infinite Impulse Response models with a special structure imposed on the infinite-dimensional vector $(a_t)_{t=1}^{\infty}$. FIR models are commonly viewed as an approximation of those models since their special structure typically imposes exponentially small coefficients for the tail of the sequence $(a_t)_{t=1}^{\infty}$; we refer to \cite{10.2307/1912559,Ljung:1999:SIT:293154} for more on this topic. In statistics, on the other hand, several authors have recently studied the non-asymptotic properties of various regression estimators and, in particular, the ones built with independent covariates. For the sample covariance matrix  with independent covariates, \citet{ADAMCZAK2011195} provides Restricted Isometric Property bounds for the sub-Gaussian case and \citet{yaskov2015} relaxes these conditions to bounds on the first $2+\alpha$ moments while still obtaining similar results which are shown to be optimal; the high-dimensional case is studied in \citet{koltchinskii2017} where comparable results are obtained for Gaussian covariate vectors while replacing their dimension by the stable rank of the covariance operator. The statistical properties of the least-squares estimator for the independent covariates case are studied in \citet{Oliveira2016} under a 4th moment assumption, and a statistically and computationally efficient algorithm is suggested and analyzed in \citet{JMLR:v18:16-335}. In general, a dependence structure --such as the one stemming from the dynamics of FIR models-- is detrimental to the behavior of estimators, as noted by \citet{dalalyan2017}, who studied the performance of the Lasso estimator under correlated design conditions. In the present work, we show that this is not the case for the least-squares estimator in the FIR model and that the obtained statistical properties are close to the ones obtained in the independent case. This is mainly due to the fact that for FIR models the  covariates are uncorrelated although they are dependent.

In Section 2, we introduce some notation and preliminaries related to the assumptions under which the results are obtained. In Section 3 we display the main results of the paper, Theorems \ref{thm: estimation error}, \ref{thm: oracle inequality} and \ref{thm: risk bound}  which provide high probability bounds on the estimation error and the prediction risk of the least-squares estimator based on two concentration inequalities on the norm of sums of dependent covariate vectors derived in Theorem \ref{thm: Bound on the multiplier process} and on the singular values of their covariance operator derived in Theorem \ref{thm: bound on the spectrum}, which are of independent value on their own. Section 3 provides a discussion on the optimality of the obtained results based on a Cram\`er-Rao risk lower bound (CRLB) in the Gaussian case. Finally, proofs of the two concentration inequalities (Theorems  \ref{thm: Bound on the multiplier process} and \ref{thm: bound on the spectrum}) are displayed in an appendix at the end of the paper.   

\section{Preliminaries and notation}
We denote by $(\Omega,\mathcal{F},\mathbb{P})$ the underlying probability space and by $\E$ the corresponding expectation operator. All norms will be distinguished by a subscript denoting the underlying normed space, except for the standard absolute value $|\cdot|$. The usual $p$-norms of a vector $x \in \mathbb{R}^n$ are denoted by
\begin{equation*}
    |x|_{p} = \left(\sum_{i=1}^n |x_i|^p \right)^{1/p} \quad \textrm{for }1 \leq p < \infty, \qquad |x|_{\infty} = \max \limits_{i \in \llbracket  1,n \rrbracket} |x_i|.
\end{equation*}
Denote by $l_p(\mathbb{Z})$ the standard $l_p$-spaces for infinite sequences of $\mathbb{R}^\mathbb{Z}$, the associated norms are 
\begin{equation*}
    |x|_{l_p(\mathbb{Z})} = \left(\sum_{i=-\infty}^\infty x_i^p \right)^{1/p} \quad \textrm{for}\,\,1 \leq p < \infty, \qquad |x|_{l_\infty(\mathbb{Z})} = \sup \limits_{i \in \mathbb{Z}} |x_i|.
\end{equation*}
The $p$-Schatten norm of an $N \times M$ matrix $A$
in the Schatten space $S_p$
is denoted by 
\begin{equation*}
    |A|_{S_p} = (\tr(A^\top A)^{p/2})^{1/p} \quad \textrm{for }1 \leq p < \infty, \qquad |A|_{S_\infty} = \max \limits_{|x|_2 \leq 1} |Ax|_2.  
\end{equation*}
Operators from $\mathbb{R}^\mathbb{Z}$ to $\mathbb{R}^\mathbb{Z}$ are identified with their corresponding infinite matrices and their operator norm will be denoted $|\cdot|_{2 \to 2}$.

The spaces  $L_p(\Omega,\mathcal{F},\mathbb{P}) = \{ x\colon \Omega \to \mathbb{R} \textrm{ measurable}\colon \ \E(|x|^p) < \infty  \}$ of random variables are endorsed with the $L_p$-norms
\begin{equation*}
    |x|_{L_p} = (\E(|x|^p))^{1/p} \quad \textrm{for}\,\, 1 \leq p < \infty, \qquad |x|_{L_\infty} = \ess \sup \limits_{\omega \in \Omega} |x(\omega)|,  
\end{equation*}
where $\ess \sup$ refers to the essential supremum with respect to the measure $\mathbb{P}$.

We study the least squares estimator in the context of 'sub-Gaussian random variables'. To introduce this class of random variables, we need to describe the tail behavior of a random variable. Define, for $0 <\alpha < \infty$, the function  
\begin{equation*}
    \psi_\alpha(s) = \exp(s^\alpha)-1, \quad s\ge 0.
\end{equation*}
The Orlicz norm $| \cdot |_{\psi_\alpha}$ of a random variable $x$ is defined as
\begin{equation*}
    |x|_{\psi_\alpha} = \inf\lbrace c >0\colon \E(\psi_\alpha (|x|/c)) \le 1 \rbrace,
\end{equation*}
and the corresponding Orlicz space is 
\begin{equation*}
     L_{\psi_\alpha}(\Omega,\mathcal{F},\mathbb{P}) = \{ x\colon \Omega \to \mathbb{R} \textrm{ measurable}; \ |x|_{\psi_\alpha} < \infty  \}.
\end{equation*}
Note that we recover the definition of $L_p$ space by taking $\psi_p (s)= s^p$ in the definition of $\psi_\alpha$. $L_{\psi_2}(\Omega,\mathcal{F},\mathbb{P})$ is the space of sub-Gaussian random variables, \emph{i.e.}, whose distribution exhibits a tail behavior similar to that of the normal distribution. $L_{\psi_1}(\Omega,\mathcal{F},\mathbb{P})$ is the space of sub-exponential random variables, whose tail behavior is similar to that of an exponential random variable. For more information on Orlicz spaces we refer to \cite{doi:10.1137/1005082}.

A random vector $X \in \mathbb{R}^n$ is sub-Gaussian if for all $u \in \mathbb{R}^n$ the marginal $\inner{X}{u}$ is a sub-Gaussian random variable. The sub-Gaussian norm of $X$ in this case is
\begin{equation*}
    |X|_{\psi_2} = \sup \limits_{u \in \mathbb{R}^n,|u|_2= 1} |\inner{u}{X}|_{\psi_2}. 
\end{equation*}

Finally, we recall a few concepts from the generic chaining literature \cite{talagrand2006generic,gine_nickl_2015}. Let $(\mathcal{A},d)$ be a metric space. The distance of a point $t \in \mathcal{A}$ to a subset $\mathbb{A} \subseteq \mathcal{A}$ is defined as
\begin{equation*}
    d(t,\mathbb{A}) = \inf \limits_{s \in \mathbb{A} } d(t,s).
\end{equation*}
The diameter of the set $\mathbb{A}$ is 
\begin{equation*}
    \Delta(\mathbb{A}) = \sup \limits_{(s,t) \in \mathbb{A}^2} d(t,s),
\end{equation*}
and the covering number $N(\mathcal{A},d,u)$ is the smallest number of balls in $(\mathcal{A},d)$ of radius less than $u$  needed to cover $\mathcal{A}$ (\emph{i.e.}, whose union includes $\mathcal{A}$). A ball of center $c \in \mathcal{A}$ and radius $r\ge 0$ with respect to a distance $d$ or a metric $|\cdot|_d$ will be denoted $B_d(c,r)$ or $B_{|\cdot|_d}(c,r)$, respectively.

The gamma-$\alpha$ functional $\gamma_\alpha(\mathcal{A},d)$ for the metric space $(\mathcal{A}, d)$ and its corresponding upper bound by the Dudley chaining integral are defined as follows: 
\begin{align}\label{ekv}
    \gamma_\alpha(\mathcal{A},d) = \inf \sup \limits_{t \in \mathcal{A}} \sum \limits_{r= 0}^{\infty} 2^{r/\alpha} d(t,\mathbb{A}_r) 
    \lesssim \int_{0}^{ \Delta(\mathcal{A})} (\ln N(\mathcal{A}, d,u))^{1/\alpha}du.
\end{align}
Where the infimum is taken over all sequences of sets $(\mathbb{A}_r)_{r \in \mathbb{N}}$ in $\mathcal{A}$ with $|\mathbb{A}_0|=1$ and $|\mathbb{A}_r| \leq 2^{2^r}$ (\cite{talagrand2006generic}).

Throughout the paper, $C$, $c$, $c_1$, $\dots$ will denote positive constants whose exact values are not important for the derivation of the results and which may change from one line to another. $x \lesssim y$ is a shorthand notation for the statement 'there exists a positive constant $c$ such that $x \leq cy$', and $x \simeq y$ means that $x \lesssim y$ and $y \lesssim x$. The minimum (maximum) of two real numbers $x$ and $y$ is denoted as $\min(x,y)= x \wedge y$ ($\max(x,y)= x \vee y$).

\section{Results}
Let the integer $T \geq 1$ be a fixed time lag. We are interested in  estimating the parameters of an FIR model, based on a single sample from the trajectory of the FIR system, denoted by 
$$
S =(x,y) = ((x_t)_{t=2-T}^{\infty},(y_t)_{t=1}^{\infty}),
$$
over the interval $\llbracket  2-T, \dots, N\rrbracket$. 

\noindent
The FIR sample satisfies the relation (\emph{cf.} \eqref{eq:time_shifted})
\begin{equation} \label{FIR systems}
    y_t = f(X_t) +\varepsilon_t, \qquad t \in  \llbracket  1,N \rrbracket. 
\end{equation}
where $f\colon \mathbb{R}^T \to \mathbb{R}$ is an unknown regressor in the general additive system case. In the linear case, the FIR system is described by 
\begin{equation} \label{linear FIR systems}
    y_t = a^\top X_t +\varepsilon_t, \qquad t \in  \llbracket  1,N \rrbracket, 
\end{equation}
where $a$ is an unknown parameter vector in $\mathbb{R}^T$. We will impose two different sets of conditions on the independent centered disturbances $\varepsilon_1, \varepsilon_2, \dots, \varepsilon_N$:
\begin{itemize}
    \item In Theorem \ref{thm: estimation error} and \ref{thm: oracle inequality} the disturbances are sub-Gaussian with common $\psi_2$-norm upper bound: 
    $$| \varepsilon_t |_{\psi_2} \leq L_\varepsilon, \quad t \in  \llbracket  1,N \rrbracket. $$
    \item In Theorem \ref{thm: risk bound} we relax these conditions to the existence of a common variance upper bound: 
    $$\E (\varepsilon_t^2)  \leq \sigma_\varepsilon^2, \quad t \in  \llbracket  1,N \rrbracket. $$
\end{itemize}

Similarly, we assume that the random variables $x_t$ are centered, unit variance with a common $\psi_2$-norm upper bound 
$$
| x_t |_{\psi_2} \leq L_x, \quad t \in  \llbracket  2-T,N \rrbracket.
$$ 
We define the least squares estimator $\hat{a}$ of the parameter vector $a$ of the linear FIR system \eqref{linear FIR systems} by
\begin{equation} \label{least square estimate}
    \hat{a} = (X^\top X)^\dagger X^\top y =(\hat{\Sigma})^\dagger \frac{1}{N}X^\top y.
\end{equation}
Here, $A^\dagger$ denotes the Moore-Penrose pseudo-inverse of a square matrix $A$. In the first main result of the paper, Theorem \ref{thm: estimation error} below, we derive a non-asymptotic high probability upper bound for the performance of $\hat{a}$ in terms of the square error $|\hat{a}  - a|_2^2$.
\begin{theorem} \label{thm: estimation error}
Let $X_1, \dots, X_N$ be the time shifted covariates (\emph{cf.} \eqref{eq:time_shifted}) of the linear FIR system \eqref{linear FIR systems} with independent centered sub-Gaussian entries of unit variance and common $\psi_2$-norm upper bound $L_x$, and assume that the noise entries $(\varepsilon_t)_{t=1}^N$ have a common $\psi_2$-norm upper bound $L_\varepsilon$. Then, there is an absolute constant $C > 0$ such that, for all $\delta \in (0,1)$, $\eta \in (0,\ 2e^{-1}]$ and as long as
\begin{equation*}
    N \geq C (L_x^2 \vee L_x^4) \frac{T\ln(T \vee \eta^{-1})}{\delta^2},
\end{equation*}
the least squares estimator $\hat{a}$ of $a$ satisfies
\begin{equation*}
     | \hat{a}  - a|_2^2 \leq \frac{2}{(1-\delta)^2}\frac{T}{N} \left(1 +  CL_x^2L_\varepsilon^2 \ln^2\frac{2}{\eta} \right),
\end{equation*}
with probability at least $1-\eta$.
\end{theorem}
Moreover, we can make explicit the sample complexity, \emph{i.e.}, the sample size $N$ required to obtain a reliable estimate of the square error with some desired precision $\epsilon >0 $ and probability at least $1 - \eta$. So for an estimation error $| \hat{a}  - a|_2^2 \leq \epsilon$, it enough to choose 
\begin{equation*}
     N \geq C \frac{1}{\epsilon}\left( \frac{T\ln^2(\eta^{-1})}{(1-\delta)^2}  \vee \frac{T\ln(T \vee \eta^{-1})}{\delta^2} \right),
\end{equation*}
for some constant $C>0$ depending on $L_x $and $L_\epsilon$. Next, we show that a similar result holds up to a bias term, for the general case with an unknown regressor $f\colon \mathbb{R}^T\to \mathbb{R}$. \\
Given $a \in \mathbb{R}^T$, define the prediction loss by $\mathcal{L}(a)$
\begin{equation}\label{a-opt} 
\mathcal{L}(a) = \frac{1}{N}|y-Xa|_2^2.
\end{equation}
Since $\hat{a}$ is the least square estimator, it satisfies the orthogonality condition
\begin{equation}\label{eq: orthogoality condition}
    X^\top (y-X\hat{a})=0.
\end{equation}

In the next main result of the paper, Theorem \ref{thm: oracle inequality}, we give a high probability bound for the prediction loss in the miss-specified case. 
\begin{theorem} \label{thm: oracle inequality}
Let $X_1, \dots, X_N$ be the time shifted covariates of the general FIR system \eqref{FIR systems} with independent centered sub-Gaussian entries of unit variance and common $\psi_2$-norm upper bound $L_x$, and assume that the noise entries  $(\varepsilon_t)_{t=1}^N$ have unit variance with common $\psi_2$-norm upper bound $L_\varepsilon$. Then, there is an absolute constant $C > 0$ such that, for all $\delta \in (0,1)$, $\eta \in (0,\ 2e^{-1}]$ and as long as
\begin{equation*}
    N \geq C (L_x^2 \vee L_x^4) \frac{T\ln(T \vee \eta^{-1})}{\delta^2},
\end{equation*}
the miss-specification prediction loss satisfies
\begin{equation*}
     \mathcal{L}(\hat{a}) \leq \inf_{a \in \mathbb{R}^T} \mathcal{L}(a) + \frac{2}{1-\delta}\frac{T}{N} \sigma_\varepsilon\left(1 +  CL_x^2L_\varepsilon^2 \ln^2\frac{2}{\eta}  \right),
\end{equation*}
with probability at least $1-\eta$.
\end{theorem}

Now, let us evaluate the prediction performance of our estimates $X\hat{a}$. Suppose we are in the general case with an unknown regressor $f\colon \mathbb{R}^T\to \mathbb{R}$. Given $a \in \mathbb{R}^T$, define the prediction risk $\mathcal{R}(a)$ and its oracle minimizer $a_{\min}$ as
\begin{equation}\label{a-opt} 
\mathcal{R}(a) = \E(\mathcal{L}(a)|X)= \E(\frac{1}{N}|y-Xa|_2^2 \,|\, X) \quad \textrm{and} \quad a_{\min} \in \arg \min \limits_{a \in \mathbb{R}^T} \mathcal{R}(a),
\end{equation}
Since $x$ is sub-Gaussian and $\varepsilon$ has finite variance, $\E(y^2) < \infty$ and $\E(\hat{\Sigma}) = I_T$, the risk measure $\mathcal{R}(a)$ is well defined and is finite for all $a \in \mathbb{R}^T$. 
The normal equation of the minimum norm problem \eqref{a-opt} is
\begin{equation} \label{eq: normal equation}
    \E((X^\top(y - X a_{\min})|X)=0.
\end{equation}
In the next main result of the paper, Theorem \ref{thm: risk bound}, we give 
a high probability upper bound for the  oracle error $\mathcal{R}(\hat{a}) - \mathcal{R}(a_{\min})$. 
\begin{theorem} \label{thm: risk bound}
Let $X_1, \dots, X_N$ be the time shifted covariates of the general FIR system \eqref{FIR systems} with independent centered sub-Gaussian entries of unit variance and
common $\psi_2$-norm upper bound $L_x$, and assume that the noise entries  $(\varepsilon_t)_{t=1}^N$ have a common variance upper bound $\sigma_\varepsilon^2$. Then, there is an absolute constant $C > 0$ such that, for all $\delta \in (0,1)$, $\eta \in (0,e^{-1}]$ and as long as
\begin{equation*}
    N \geq C (L_x^2 \vee L_x^4) \frac{T\ln(T \vee \eta^{-1})}{\delta^2},
\end{equation*}
the misspecification error satisfies
\begin{equation*}
     0\leq \mathcal{R}(\hat{a}) - \mathcal{R}(a_{\min}) \leq \frac{1 + \delta}{1 - \delta}\frac{T}{N}\sigma_\varepsilon^2.
\end{equation*}
with probability at least $1-\eta$.
\end{theorem}

These last three Theorems extend the work on least squares for the case of independent covariates to the time shifted covariate case. In the independent covariates case, bounds are derived in \citet{audibert2011} and \citet{pmlr-v23-hsu12} under the sub-Gaussian noise condition and with probability larger than $\frac{1}{N}$ i.e non uniformly on the probability. In our case, we only need a finite noise variance condition to show a bound for any success probability. A break through result in the independent covariates case is the one obtained by \citet{Oliveira2016}, where Under a $4$th moment assumption and non uniformly on the probabilities, he derives similar results using Fuk-Nagaev bound by \citet{10.2307/20161536}.

The proofs of Theorems \ref{thm: estimation error}, \ref{thm: oracle inequality} and \ref{thm: risk bound} are based on two upper bound estimates. The first one is an estimate of the operator norm distance of the sample covariance to the identity operator $\left| \hat{\Sigma} - I_T \right|_{S_\infty}$, derived in Theorem \ref{thm: bound on the spectrum} below, in the form of Restricted Isometric Property (cf. \cite{Candes:2006:NSR:2263452.2272610}). The second one is an estimate of the $L_2$-norm of the multiplier process for the  FIR design covariates, $\frac{1}{N}X^\top \varepsilon$, given in Theorem \ref{thm: Bound on the multiplier process} below. Multiplier processes are important as they appear quite naturally in many high dimensional statistics' problems. The term 
\begin{equation*}
\left| \frac{1}{N}X^\top \varepsilon \right|_2 = \frac{1}{N} \sup_{|w|_2} \sum_{k =1}^N \varepsilon_k\innerl{w}{X_k} - \E(\varepsilon_k\innerl{w}{X_k})    
\end{equation*}
is seen in this way as the supremum of an empirical process yet it is not a multiplier process in the standard sense (see e.g. \cite{Mendelson2017}) since the $X_i$'s are dependent. 

\medskip

\begin{theorem}\label{thm: bound on the spectrum}
Let $X_1, \dots, X_N$ be the time shifted covariates of an FIR model with independent centered sub-Gaussian entries of common $\psi_2$-norm upper bound $L_x$. Then, for every $u \geq 1$ and $N \geq 2T-2$, we have 
\begin{equation} \label{eq: sup norm bound on the sample covariance}
        \left| \hat{\Sigma} - I_T \right|_{S_\infty} \lesssim  L_x^2\left(\frac{T\ln(T)}{N} + \sqrt{\frac{T\ln(T)}{N} }  + \frac{T}{N}u + \sqrt{\frac{T}{N} } u^{1/2} \right) .
\end{equation}
with probability at least $1-e^{-u}$.

In particular, we have the following bounds on the highest and smallest singular values $s_{\max} (\hat{\Sigma})$ and $s_{\min} (\hat{\Sigma})$ of the sample covariance matrix $\hat{\Sigma}$: there is an absolute constant $C > 0$ such that, for all $\delta \in (0,1)$, $\eta \in (0,e^{-1}]$ and
\begin{equation*}
    N \geq C (L_x^2 \vee L_x^4) \frac{T\ln(T \vee \eta^{-1})}{\delta^2},
\end{equation*}
we have
\begin{equation}\label{bound on the spectrum}
     1- \delta \leq s_{\min} (\hat{\Sigma}) \leq s_{\max} (\hat{\Sigma}) \leq 1+ \delta,
\end{equation}
with probability at least $1- \eta$.
\end{theorem}
The conditions $N \geq 2T-2$ and $u \geq 1$ for deriving \eqref{eq: sup norm bound on the sample covariance} and $\eta \leq e^{-1}$ for deriving \eqref{bound on the spectrum} are mainly technical and are introduced to simplify the proof. This result extends the work by \citet{meckes2007} on random square symmetric Toeplitz matrices ($N = T$) with independent entries satisfying a log-Sobolev condition to the rectangular case with sub-Gaussian design matrix $X$. It also complements the work by \citet{10.1007/978-3-0348-0490-5_16}, who proves the following bound on the expectation of a random rectangular Toeplitz matrix $X$:
\begin{equation*} 
    \E (| X |_{S_\infty} ) \lesssim \sqrt{N} + \sqrt{T\log (T)}.
\end{equation*}

An optimal result for the case of i.i.d Gaussian covariates has been obtained by \citet{koltchinskii2017} and shows that with probability at least $1-e^{-u}$:
\begin{equation*} 
    \left| \hat{\Sigma} - I_T \right|_{S_\infty} \lesssim   \frac{T}{N}  + \sqrt{\frac{T}{N} }  + \frac{T}{N}u + \sqrt{\frac{T}{N} } u^{1/2} .
\end{equation*}
This raises the question of the necessity of the having $\ln(T)$ factors in Theorem \ref{thm: bound on the spectrum}.

\begin{theorem}\label{thm: Bound on the multiplier process}
Let $X_1, \dots, X_N$ be the time shifted covariates of an FIR model with independent centered sub-Gaussian entries of common $\psi_2$-norm upper bound $L_x$, and assume that the noise entries $(\varepsilon_t)_{t=1}^N$ have unit variance with common $\psi_2$-norm upper bound $L_\varepsilon$. Then, for every $\eta \in (0, 1)$, we have 
\begin{equation}\label{Bound on the multiplier process}
    \left| \frac{1}{N}X^\top \varepsilon \right|_2 \leq \sqrt{\frac{T}{N}} \sigma_x \sigma_\varepsilon \left[1+c L_\varepsilon L_x \left(\ln\frac{1}{\eta} \vee \ln^\frac{3}{4}\frac{1}{\eta}\right) \right],
\end{equation}
with probability at least $1-\eta$.
\end{theorem}

The proof of these last two Theorems is deferred to the appendix. 


\begin{proof}[Proof of Theorem \ref{thm: estimation error}]
We have
\begin{align}
        | \hat{a}  - a|_2  &= | (X^\top X)^\dagger X^\top y  - a|_2 \nonumber 
        = | (X^\top X)^\dagger X^\top (Xa + \varepsilon)  - a|_2 \nonumber  \\
        &= | (X^\top X)^\dagger X^\top \varepsilon|_2 \nonumber \leq s_{\max} ((\hat{\Sigma})^\dagger)  \left|\frac{1}{N}X^\top \varepsilon \right|_2 \nonumber \\
        &= \frac{1}{s_{\min}(\hat{\Sigma})}   \left|\frac{1}{N}X^\top \varepsilon \right|_2. \label{eq: prediction error decomposition}
\end{align}
Replacing $\eta$ with $\eta/2$ in Theorems~\ref{thm: bound on the spectrum} and \ref{thm: Bound on the multiplier process}, a simple union bound shows that, for 
\begin{equation*}
    N \geq C (L_x^2 \vee L_x^4) \frac{T\ln(T \vee 2\eta^{-1})}{\delta^2},
\end{equation*}
the matrix $\hat{\Sigma}$ is invertible with probability at least $1-\eta$. Furthermore, on this event, the bounds in equations \eqref{bound on the spectrum} and \eqref{Bound on the multiplier process} hold, yielding upper bounds for both the inverse of the least singular value, namely $  \frac{1}{s_{\min}(\hat{\Sigma})} \leq \frac{1}{1-\delta}$, and the multiplier process $|\frac{1}{N}X^\top \varepsilon |_2$ which together imply the claimed result.
\end{proof}

\begin{proof}[Proof of Theorem \ref{thm: oracle inequality}]
The loss miss-specification bound is obtained as follows.
\begin{align}
    \mathcal{L}(\hat{a}) - \mathcal{L}(a)&= \frac{1}{N}|X\hat{a}-f|_2^2- \frac{1}{N}|Xa-f|_2^2\nonumber \\
    &= \frac{1}{N}\innerl{X\hat{a}-Xa}{X\hat{a}+Xa-2f} \nonumber  \\
    &= -\frac{1}{N}|X\hat{a}-Xa|_2^2+\frac{2}{N}\innerl{X\hat{a}-Xa}{X\hat{a}-f}. \label{eq: loss decomposition}
\end{align}
Using the orthogonality condition in \eqref{eq: orthogoality condition} we get, 
\begin{align*}
   \mathcal{L}(\hat{a}) &= \mathcal{L}(a)  -\frac{1}{N}|X\hat{a}-Xa|_2^2+\frac{2}{N}\innerl{\hat{a}-a}{X^\top \varepsilon}\\
   &\leq \mathcal{L}(a) -\frac{1}{N}|X(\hat{a}-a)|_2^2+2|\hat{a}-a|_2\left|\frac{X^\top \varepsilon}{N}\right|_2\\
   &\leq \mathcal{L}(a) -s_{\min}(\hat{\Sigma})|\hat{a}-a|_2^2+2|\hat{a}-a|_2\left|\frac{X^\top \varepsilon}{N}\right|_2 \\
   &\leq \mathcal{L}(a) + \frac{1}{s_{\min}(\hat{\Sigma})}\left|\frac{X^\top \varepsilon}{N}\right|_2^2,
\end{align*}
where in the last inequality we used the elementary identity $-ab^2+2cb \leq c^2/a$. From Theorem \ref{thm: bound on the spectrum} and \ref{thm: Bound on the multiplier process} we conclude that  for a sample size 
\begin{equation*}
    N \geq C (L_x^2 \vee L_x^4) \frac{T\ln(T \vee 2\eta^{-1})}{\delta^2},
\end{equation*}
and on an event of probability at least $1-\eta$ both $s_{\min}(\hat{\Sigma}) \geq 1-\delta$ and
\begin{equation}
    \left| \frac{1}{N}X^\top \varepsilon \right|_2 \leq  \sqrt{\frac{T}{N}}  \sigma_\varepsilon \left[1+C L_\varepsilon L_x \left(\ln\frac{2}{\eta} \vee \ln^\frac{3}{4}\frac{2}{\eta}\right) \right]
\end{equation}
hold. These two result combined with the condition $\eta \in (0,\ 2e^{-1}]$ give the conclusion of the Theorem.
\end{proof}


\begin{proof}[Proof of Theorem \ref{thm: risk bound}]
We place ourselves again on an event of probability at least $1-\eta$ obtained from Theorem \ref{thm: bound on the spectrum} with  sample size satisfying 
\begin{equation*}
    N \geq C (L_x^2 \vee L_x^4) \frac{T\ln(T \vee \eta^{-1})}{\delta^2}.
\end{equation*}
The Prediction risk bound is obtained by taking the conditional expectation given $X$ on equation \eqref{eq: loss decomposition}:
\begin{align*}
    \mathcal{R}(\hat{a}) - \mathcal{R}(a) =\frac{1}{N} \E\left(-|X\hat{a}-Xa|_2^2|\, X\right)+\frac{2}{N}\E\left(\innerl{X\hat{a}-Xa}{X\hat{a}-f}|\, X\right). 
\end{align*}
Using the normal equation \eqref{eq: normal equation} we obtain,
\begin{align*}
   \mathcal{R}(\hat{a}) &= \mathcal{R}(a)  -\frac{1}{N}\E\left(|X\hat{a}-Xa|_2^2 - 2\innerl{\hat{a}-a}{X^\top \varepsilon}|\, X\right)\\
   &\leq \mathcal{R}(a)  -\E\left(\frac{1}{N}|X(\hat{a}-a)|_2^2-2|\hat{a}-a|_2\left|\frac{X^\top \varepsilon}{N}\right|_2\Big|\, X\right)\\
   &\leq \mathcal{R}(a)  -\E\left( s_{\min}(\hat{\Sigma})|\hat{a}-a|_2^2 -2|\hat{a}-a|_2\left|\frac{X^\top \varepsilon}{N}\right|_2\Big|\, X\right). 
\end{align*}
On this event we have $1-\delta \leq s_{\min}(\hat{\Sigma})) \leq s_{\max}(\hat{\Sigma})) \leq 1+\delta$ which in tern implies,
\begin{align}
   \mathcal{R}(\hat{a}) &\leq \mathcal{R}(a) +\E\left(-(1 - \delta)|\hat{a}-a|_2^2+2|\hat{a}-a|_2\left|\frac{X^\top \varepsilon}{N}\right|_2\Big|\, X\right)\nonumber \\
   &\leq  \mathcal{R}(a) +\frac{1}{1 - \delta}\E\left(\left|\frac{X^\top \varepsilon}{N}\right|_2^2\Big|\, X\right), \label{eq: risk bound main inequality}
\end{align}
where, again, we used the elementary identity $-ab^2+2cb \leq c^2/a$ in the last inequality. The conditional expectation on the other hand is given by 
\begin{align*}
    \E\left(\left|\frac{X^\top \varepsilon}{N}\right|_2^2\Big|\, X\right) &= \frac{1}{N} \E\left(\frac{\tr (\varepsilon^\top X X^\top \varepsilon)}{N}\Big|\, X\right) \\
    &= \frac{1}{N} \frac{\tr (\E(\varepsilon \varepsilon^\top) X X^\top )}{N} \leq \frac{\tr(\hat{\Sigma})}{N}\sigma_\varepsilon^2,
\end{align*}
which together with the risk bound in \eqref{eq: risk bound main inequality} and the fact that $s_{\max}(\hat{\sigma})\leq 1+\delta$ yield
\begin{align*}
   \mathcal{R}(\hat{a}) - \mathcal{R}(a) \leq \frac{1 + \delta}{1 - \delta}\frac{T}{N}\sigma_\varepsilon^2.
\end{align*}
which concludes the proof of the Theorem.
\end{proof}

\section{Optimality of the bounds and analysis of the Gaussian noise case}
In this section, we check the optimality of our bounds for the least squares estimator
by obtaining a matching lower bound when the noise is Gaussian. Consider the case where $(\varepsilon_t)_{t=1}^N$ is a sequence of independent identically distributed Gaussian random variables $\mathcal{N}(0,\sigma_x^2)$ and $(X_t)_{t=1}^N$ are the time-shifted design covariates defined in \eqref{eq:time_shifted} with independent centered unit variance sub-Gaussian entries of common $\psi_2$-norm upper bound $L_x$.

First note that $\Hat{a}$ is conditionally unbiased since  
\begin{align*}
    \E(\hat{a}|X)= \E ((X^\top X)^\dagger X^\top y |X)= \E((X^\top X)^\dagger X^\top (Xa + \varepsilon))|  X ) = a.
\end{align*}
Hence, we can use Cram\'er-Rao theory \cite{Rao1992} to derive a lower bound performance for such unbiased estimator in terms of its mean square error loss, defined as
\begin{equation}\label{MSE}
    \MSE (\hat{a}) = \E(|\hat{a}-a|_2^2).
\end{equation}
Since the $(y_t)_{t=1}^N$ are conditionally independent, the conditional distribution of $y=(y_t)_{t=1}^N$ given $x=(x_t)_{t=2-T}^N$, parameterized by $a$, factors as
\begin{align*}
            f_{y | x}(y|x;a) & = \prod_{t=1}^N f_{y_t|x}(y_t|x;a).  
\end{align*}
 The log-likelihood function becomes
\begin{align*}
            \mathcal{L}(Y, X, a) & = \ln f_{y | x}(y|x;a)= \sum_{t=1}^N\ln( f_{y_t|x}(y_t|x;a) ) \\   
            & = -\frac{N}{2} \ln(2\pi \sigma_\varepsilon^2) - \frac{1}{2\sigma_\varepsilon^2} \sum_{t=1}^{N} (y_t - a^\top X_t)^2. 
\end{align*}
Differentiating twice with respect to the parameter $a$, we obtain the conditional Fisher information matrix 
\begin{align*}
    \mathcal{I}_T(a | X) &= -\nabla_a^2 \mathcal{L}(Y, X, a) = \frac{1}{\sigma_\varepsilon^2} \sum_{t=1}^N X_t X_t^\top  \\
    &= \frac{1}{\sigma_\varepsilon^2}  X^\top X = \frac{N}{\sigma_\varepsilon^2} \hat{\Sigma}.
\end{align*}
The following proposition provides bounds for the expected values of the spectrum and the trace of the sample covariance matrix of FIR models and will be useful for deriving a lower bound estimate of the MSE. 
\begin{proposition}\label{prop: bounds on the expected operator norm}
The following estimate holds for the singular values of the sample covariance matrix:
\begin{align}
    \E(s_i(\hat{\Sigma})) \leq 1 + CL_x^2\left(\frac{T}{N} \ln(T) + \sqrt{\frac{T}{N} \ln(T)}  \right). 
\end{align}
Furthermore, the following lower bound for the trace of its inverse holds:
\begin{align*}
    \E(\tr(\hat{\Sigma}^\dagger)) \geq T \left\{1 -  CL_x^2\left(\frac{T}{N} \ln(T) + \sqrt{\frac{T}{N} \ln(T)} \right)\right\}.
\end{align*}

\end{proposition}

\begin{proof}
A rescaling of $(x_k)_{k=2-T}^N$ by $L_x$ shows that without loss of generality we can assume $L_ x = 1$, we also let
\begin{equation*}
    E :=\frac{T}{N} \ln(T) + \sqrt{\frac{T}{N} \ln(T)} .
\end{equation*}
From Theorem~\ref{thm: bound on the spectrum} we have, for all $t \geq 0$,
\begin{align*}
    \mathbb{P}\left(s_i(\hat{\Sigma})-1-cL_x^2 E \geq t\right) \leq  
    \begin{cases}
    1,& 0 \leq t \leq \displaystyle \sqrt{\frac{T}{N}},\\ \\
    \exp\left(-C \displaystyle \frac{Nt}{T} \right),& \sqrt{\frac{T}{N}} < t \leq \displaystyle 1,\\ \\
    \exp\left(-C \displaystyle \frac{Nt^2}{T} \right),& t > \displaystyle 1.
    \end{cases}
\end{align*}
Hence, 
\begin{align*}
    &\E\left(s_i(\hat{\Sigma})-1-cL_x^2E\right) \leq \int_{0}^{\infty} \mathbb{P}\left(s_i(\hat{\Sigma})-1-cL_x^2E \geq t\right) dt\\
    &\qquad\qquad = \int_{0}^{\sqrt{\frac{T}{N}}} dt+\int_{\sqrt{\frac{T}{N}}}^{1} 
    \exp\left(-C \frac{Nt}{T} \right) dt + \int_{1}^{\infty} \exp\left(-C \frac{Nt^2}{T} \right) dt \\
    &\qquad\qquad \lesssim \sqrt{\frac{T}{N}} + \left[ \frac{T}{N}\exp\left(-C \frac{Nt}{T} \right)\right]^{\sqrt{\frac{T}{N}}}_1 + \frac{\sqrt{\pi}}{2} \sqrt{\frac{T}{N}} \simeq \sqrt{\frac{T}{N}} + o\left(\frac{T}{N}\right),
\end{align*}
where in the last inequality we have used the following gamma integral
\begin{equation*}
    \int_{0}^{\infty} e^{-at^b}dt = \frac{1}{b}a^{-\frac{1}{b}} \Gamma(\frac{1}{b}),
\end{equation*}
valid for all $a >0$ and $b >0$. Thus after rescaling by $L_x$ we have,
\begin{align*}
    \E(s_i(\hat{\Sigma})) &\leq 1 + C\left[\frac{T}{N} \ln(T) + \sqrt{\frac{T}{N} \ln(T)}  + \sqrt{\frac{T}{N}} + o\left(\frac{T}{N}\right)\right]\\
    &\leq 1 + C\left(\frac{T}{N} \ln(T) + \sqrt{\frac{T}{N} \ln(T)} \right). 
\end{align*}
This gives the following lower bound:
\begin{align*}
    \E(\tr(\hat{\Sigma}^\dagger)) &= \E\left(\sum_{i=1}^{T} \frac{1}{s_i(\hat{\Sigma})} \right) \geq \sum_{i=1}^{T} \frac{1}{\E(s_i(\hat{\Sigma}))} \qquad \text{(by Jensen's inequality)}, \\
    &\geq \frac{T}{1 +  CL_x^2\left(\frac{T}{N} \ln(T) + \sqrt{\frac{T}{N} \ln(T)}  \right) } \\ 
    & \geq T \left\{1 -  CL_x^2\left(\frac{T}{N} \ln(T) + \sqrt{\frac{T}{N} \ln(T)}  \right)\right\},
\end{align*}
which in turn yields  the statement of the proposition.
\end{proof}

Thus, the Cram{\'e}r-Rao lower bound for a conditionally unbiased estimator is
\begin{align*}
 \cov(\hat{a}_T ) &= \E(\cov(\hat{a}_T | X)) + \cov(\E(\hat{a}_T | X)) \succeq \E( \mathcal{I}_T(a| X)^\dagger) \\
 &= \sigma_{\varepsilon}^2\E((X^\top X)^\dagger)= \frac{\sigma_{\varepsilon}^2}{N}\E(\hat{\Sigma}^\dagger ).    
\end{align*}
This implies the following lower bound on the mean square error:
\begin{align*}
    \MSE (\hat{a}) &= \E(|\hat{a}-\E\hat{a}|_2^2) +\E(|\hat{a}- a|_2^2)= \frac{\sigma_{\varepsilon}^2}{N}\E(\tr(\hat{\Sigma}^\dagger)) \\
    &\geq \frac{T\sigma_{\varepsilon}^2}{N}\left\{1 -  CL_x^2\left( \frac{T}{N} \ln(T) + \sqrt{\frac{T}{N} \ln(T)}  \right)\right\}.
\end{align*}
In the linear case, equation \eqref{eq: prediction error decomposition} in the derivation of Theorem \ref{thm: estimation error} gives us,
\begin{align*}
        | \hat{a} - a|^2_2  \leq \frac{1}{s^2_{\min}(\hat{\Sigma})}   \left|\frac{1}{N}X^\top \varepsilon \right|^2_2. 
\end{align*}
Taking the conditional expectation on $X$ gives us
\begin{align*}
        \E(|\hat{a} - a|^2_2 \,|\, X) &\leq \frac{1}{Ns^2_{\min}(\hat{\Sigma})}   \E(\frac{1}{N} \left|X^\top \varepsilon \right|^2_2\,|\, X)\\
        &= \frac{1}{Ns^2_{\min}(\hat{\Sigma})} \frac{\tr (\E(\varepsilon \varepsilon^\top) X X^\top )}{N} \leq \frac{\tr(\hat{\Sigma})}{Ns^2_{\min}(\hat{\Sigma})}\sigma_\varepsilon^2. 
\end{align*}
Theorem \ref{thm: bound on the spectrum} implies that in an event of probability at least $1-\eta$ with $\eta \in (0, e^{-1})$ and under the condition
\begin{equation*}
    N \geq C (L_x^2 \vee L_x^4) \frac{T\ln(T \vee \eta^{-1})}{\delta^2},
\end{equation*}
we have $1-\delta \leq s_{\min}(\hat{\Sigma}) \leq s_{\max}(\hat{\Sigma}) \leq 1 +\delta$, which means
\begin{equation*}
    \E(|\hat{a}-a|_2^2 \,|\, X) \leq \frac{1 + \delta}{(1 - \delta)^2}\frac{T}{N}\sigma_\varepsilon^2.
\end{equation*}

This rate matches with the rate of the obtained CRLB, so the least-squares estimator is efficient with high probability and with an asymptotic rate of $\sigma_\varepsilon^2 T/N$. Since this estimator is also conditionally unbiased this also implies that it is rate optimal in the class of conditionally unbiased estimators and that our upper bound is tight up to a correction term of the order $\frac{T^{3/2}\ln^{1/2}(T)}{N^{3/2}}$.   

\section{Appendix}
 
\begin{proof}[Proof of Theorem~\eqref{thm: bound on the spectrum}]
A rescaling of $(x_k)_{k=2-T}^N$ by $L_x$ shows that without loss of generality we can assume $L_ x = 1$. We start by defining, for $k \in \llbracket1,N+1-T\rrbracket$, the following shifted covariates:
\begin{align*}
L_k = \begin{bmatrix}
0 & \cdots & 0 & x_1 & x_2 &\cdots &x_{N+1-T} &0 &\cdots &0
\end{bmatrix}^\top,    
\end{align*}
where $x_1$ is at $k$th position.Then we define the matrices 
\begin{equation*}
    L = [L_1 L_2 \dots L_T] \qquad \textrm{and} \qquad S = X - L,
\end{equation*}
to get a decomposition $X = L + S$ where $L$ and $S$ are independent of each other and have a shifted diagonal structure. Thus, we have
\begin{equation*}
    X^\top X = L^\top L + S^\top S + S^\top L + L^\top S.
\end{equation*}
Using this decomposition the operator norm of deviation of $X^\top X$ is upper bounded by
\begin{align}
    &|X^\top X - \E (X^\top X) |_{S_\infty} \nonumber \\
    &\leq  T-1+ |L^\top L - (N+1-T)I_T|_{S_\infty} +  |S^\top S|_{S_\infty} + 2|S^\top L|_{S_\infty}. \label{eq: 3 terms to upper-bound}
\end{align}

The proof will proceed by deriving high probability bounds for the last three terms while showing that the contribution of the last two is negligible in comparison to the first both in terms of magnitude and probability. 

We start with the second  term $|S |_{S_\infty}$. Since $N \geq 2T-2$ the matrix $S$ has two separate parts, so we define the lower and upper triangular parts of $S$ respectively by 
\begin{equation*}
S_l = \begin{bmatrix}
x_{N-T+2} &0 &\dots  &0 \\
x_{N+1-T} &\ddots &\ddots  &\vdots \\
\vdots &\ddots &\ddots &0 \\
x_N &\dots &x_{N+3-T} &x_{N+2-T}
\end{bmatrix} \quad \textrm{and} \quad    S_u = \begin{bmatrix}
x_0  &x_1 &\dots &x_{2-T} \\
0 &\ddots  &\ddots &\vdots \\
\vdots &\ddots &\ddots  &x_1 \\
0 &\dots &0 &x_0 
\end{bmatrix}.
\end{equation*}
Hence we obtain the following
\begin{equation} \label{eq: 2 norms of random vectors}
    |S |_{S_\infty} = |S_l |_{S_\infty} + |S_u |_{S_\infty} = \left(\sum \limits_{j=N+2-T}^{N} x_j^2\right)^{1/2} + \left(\sum \limits_{j=2-T}^{0} x_j^2\right)^ {1/2}.
\end{equation}
The bound reduces to a bound on the norms of two random vectors which can be estimated using Bernstein's inequality for sub-Exponential  random variables: for all $u \geq 1$ 
\begin{equation*}
    \mathbb{P}\left(\left|\left(\sum_{j=N+2-T}^{N} x_j^2\right)^ {1/2} - \sqrt{T-1} \right| \geq c u \right) \leq 2 \exp{\left( -u^2 \right)}.
\end{equation*}
A similar inequality holds for the last term in \eqref{eq: 2 norms of random vectors}. Using a union bound we get with probability at least $1- \frac{1}{3}e^{-u}$ for all $u \geq 1$:
\begin{equation}
    |S |_{S_\infty} = |S_l |_{S_\infty} + |S_u |_{S_\infty} \lesssim \sqrt{T} + \sqrt{u}.
\end{equation}
That is with the same probability:
\begin{equation}\label{eq : main HDP1}
    |S^\top S |_{S_\infty} \lesssim T + u.
\end{equation}
We will bound the second term by relating the operator norm of random Toeplitz matrices to the supremum of a multiplication process, a strategy that appeared first in the seminal work of \citet{meckes2007} and was further developed by \citet{10.1007/978-3-0348-0490-5_16}.

Since the columns of $L$ are shifted versions of each others, we have 
\begin{align*}
L^\top L=\begin{bmatrix}
\innerl{L_1}{L_1} &\innerl{L_1}{L_2} &\innerl{L_1}{L_3} &\dots &\innerl{L_1}{L_T} \\
\innerl{L_1}{L_2} &\innerl{L_1}{L_1} &\innerl{L_1}{L_2} &\ddots &\innerl{L_2}{L_T} \\
\innerl{L_1}{L_3} &\innerl{L_1}{L_2} &\innerl{L_1}{L_1} &\ddots &\innerl{L_3}{L_T} \\
\vdots &\ddots &\ddots &\ddots &\vdots\\
\innerl{L_1}{L_T} &\innerl{L_2}{L_T} &\innerl{L_3}{L_T} &\dots &\innerl{L_1}{L_1}
\end{bmatrix}. 
\end{align*}
Define the  Toeplitz operator $\mathcal{T}:\, l_2(\mathbb{Z}) \to l_2(\mathbb{Z})$ by the infinite matrix 
\begin{equation*}
    \mathcal{T} = \left[\innerl{L_j}{L_k} \mathds{1}_{|j-k|\leq T-1}\right]_{(j,k)\in \mathbb{Z}^2} - (N+1-T)I_\mathbb{Z}.
\end{equation*}
The corresponding multiplication polynomial defined for $x \in [0,1]$ is given by 
\begin{equation*}
    p(x) = \mathcal{T}_0 + 2\sum \limits_{l=1}^{T-1} \mathcal{T}_l \cos{(2 \pi l x)}.
\end{equation*}
Since $L^\top L - (N+1-T)I_T$ is a submatrix of $\mathcal{T}$, we have
\begin{equation}\label{eq : main HDP2}
    |L^\top L - (N+1-T)I_T|_{S_\infty} \leq |\mathcal{T}|_{2 \to 2} = \sup \limits_{x \in [0,1]} |p(x)|,
\end{equation}
where $|\cdot|_{2 \to 2}$ stands for the operator norm from $l_2(\mathbb{Z})$ to $l_2(\mathbb{Z})$.
Rewriting the last polynomial $p(x)$ we obtain the following upper bound for the supremum
\begin{equation*}
    \sup \limits_{x \in [0,1]} |p(x)| =  \sup \limits_{x \in [0,1]} \left|\sum \limits_{j=1}^{N+1-T} (x_j^2-1)+\sum \limits_{0\leq j < k \leq N+1-T} x_j x_k \mathds{1}_{|j-k|\leq T-1} \cos{(2\pi|j-k| x)}\right|. 
\end{equation*}
Consider the $(N+1-T) \times (N+1-T)$ matrix $H$ with $(j,k)$ entries 
\begin{equation*}
    \mathcal{H}_{jk}(x) = 1 + \mathds{1}_{|j-k|\leq T-1} \cos{(2\pi|j-k| x)}.
\end{equation*}
Using a version of Hanson-Wright inequality essentially due to \citet{rudelson2013} we end up with the following inequality: for all $u \geq 0$, 
\begin{equation*}
    \mathbb{P}\left( |p(x)-p(y)|\geq \sqrt{u} d_{S_2} (\mathcal{H}(x),\mathcal{H}(y)) + u d_{S_\infty}  (\mathcal{H}(x),\mathcal{H}(y)) \right) \leq 2 \exp{(-cu)}.
\end{equation*}
This process is a mixed tail process with the pseudo-metric spaces defined on $[0,1]$ by both $| \mathcal{H}(\cdot) |_{S_2}$ and $| \mathcal{H}(\cdot) |_{S_\infty}$. The generic chaining result proved independently by~\citet[Theorem ~$2.2.23$]{talagrand2006generic} and ~\citet[Theorem ~$3.5$]{dirksen2015} provides us with the following bound for the supremum of such mixed tail process for $u \geq 1$:
\begin{equation}\label{eq: HDP0 on polynomial}
    \mathbb{P}\left( \sup \limits_{x \in [0,1]} |p(x)| \geq C \left( E +\sqrt{u} V + u U \right)\right) \leq 2 \exp{(-u)},
\end{equation}
where 
\begin{align*}
    E &= \gamma_2([0,1],| \mathcal{H}(\cdot) |_{S_2}) + \gamma_1([0,1],| \mathcal{H}(\cdot) |_{S_\infty}), \\
    V &= \Delta_{S_2} (\mathcal{H}^{-1}([0,1])),\quad
    U= \Delta_{S_\infty}  (\mathcal{H}^{-1}([0,1])).
\end{align*}

To conclude the proof, it suffices to estimate these three terms. Let us first start with the terms which involves the norm $|\cdot|_{S_\infty}$. An estimate of $|\mathcal{H}(x)|_{S_\infty}$ is obtained by seeing $\mathcal{H}(x)$ as a sub-matrix of an infinite Toeplitz matrix $\Tilde{\mathcal{H}}(x)$ defined by,
\begin{equation*}
    \Tilde{\mathcal{H}}(x) = \left[\mathds{1}_{|j-k|\leq T-1} \cos{(2\pi|j-k| x)}\right]_{(j,k)\in \mathbb{Z}^2} + I_\mathbb{Z}.
\end{equation*}
and the corresponding multiplication polynomial
\begin{align*}
    q(z) &= 1 + 2\sum \limits_{l=1}^{T-1} \cos{(2 \pi l x)} \cos{(2 \pi l z)}\\
    &= 1 + \sum \limits_{l=1}^{T-1} \cos{(2 \pi l (x-z))} +\cos{(2 \pi l (x+z))}.
\end{align*}
We have
\begin{align*}
    |\mathcal{H}(x)|_{S_\infty} \leq |\Tilde{\mathcal{H}}(x)|_{2 \to 2} = \sup \limits_{z \in [0,1]} |q(z)|.
\end{align*}
The radius $U$ becomes
\begin{equation}\label{eq : HDP1}
    U= \Delta_{S_\infty}  (\mathcal{H}^{-1}([0,1])) \leq 2T. 
\end{equation}
Since the cosine function is $1$-Lipschitz, we have 
\begin{align*}
    d_{S_\infty}  (\mathcal{H}(x),\mathcal{H}(y)) &= \sup \limits_{z \in [0,1]} \sum \limits_{l=1}^{T-1} \cos{(2 \pi l (x-z))} - \cos{(2 \pi l (y-z))} \\
    &\qquad \qquad +\cos{(2 \pi l (x+z))} - \cos{(2 \pi l (y + z))}\\
    &\leq 4\pi \sum \limits_{l=1}^{T-1} l|x-y| \leq 2 \pi T^2 |x-y|.
\end{align*}
The $\gamma_1$ functional is evaluated as 
\begin{align}
    &\gamma_1([0,1],| \mathcal{H}(\cdot) |_{S_\infty}) \lesssim \int_{0}^{ \Delta_{S_\infty}  (\mathcal{H}^{-1}([0,1]))} \ln N([0,1],2 \pi T^2 |\cdot|,u)du \nonumber \\
    &= \int_{0}^{2T} \ln \frac{2 \pi T^2}{u} du = 2T \ln{(2 \pi T^2)} -2T \ln(2T) + 2T \simeq T \ln(T). \label{eq : HDP2}
\end{align}
We now turn to the terms involving the norm $|\cdot|_{S_2}$. The pseudo-norm $|\mathcal{H}(x)|_{S_2}$ is
\begin{align*}
    |\mathcal{H}(x)|_{S_2} = \left(\sum \limits_{l=0}^{T-1} (N-l+1) \cos^2{(2 \pi l x)}\right)^{1/2}.
\end{align*}
The radius $V$ satisfies
\begin{equation}\label{eq : HDP3}
    V= \Delta_{S_2}  (\mathcal{H}^{-1}([0,1])) \leq \sqrt{NT}. 
\end{equation}
Again the cosine function being $1$-Lipschitz, we have 
\begin{align*}
    d_{S_2}  (\mathcal{H}(x),\mathcal{H}(y)) &= \left(\sum \limits_{l=0}^{T-1} (N-l+1) \left( \cos{(2 \pi l x)} - \cos{(2 \pi l y)} \right)^2\right)^{1/2}\\
    &\leq \frac{\pi}{3^{1/2}} \sqrt{N} (T+1)^{3/2}|x-y|.
\end{align*}
The $\gamma_2$ functional satisfies
\begin{align}
    \gamma_2([0,1],| \mathcal{H}(\cdot) |_{S_2})  &\lesssim \int_{0}^{ \Delta_{S_2}  (\mathcal{H}^{-1}([0,1]))} \ln N([0,1],\frac{\pi}{3^{1/2}} \sqrt{N} (T+1)^{3/2}|\cdot|,u)du \nonumber \\
    &= \int_{0}^{\sqrt{NT}} \left( \ln \frac{\frac{\pi}{3^{1/2}} \sqrt{N} (T+1)^{3/2}}{u}\right)^{1/2} du  \nonumber \\
    &= \frac{\pi}{\sqrt{3}} \sqrt{N} (T+1)^{3/2} \int_{\sqrt{\ln \left(\frac{\pi}{3^{1/2}} \frac{(T+1)^{3/2}}{T^{1/2}}\right)}}^{\infty}\,\, t^2 \exp(-t^2/2) du \nonumber \\
    &\lesssim  \sqrt{T N \ln(T)}, \label{eq: HDP4}
\end{align}
where in the last step we did and integration by parts and used ~\cite[Formula~$7.1.13$]{Abramowitz:1974:HMF:1098650}.

\noindent Putting this last result together with \eqref{eq : HDP1}, \eqref{eq : HDP2} and \eqref{eq : HDP3} in \eqref{eq: HDP0 on polynomial} enables us to bound the supremum of the stochastic polynomial $p(x)$ with probability at least $1-\frac{1}{3} \exp{(-u)}$ for $u \geq 1$ and with a possibly greater $C$:
\begin{equation*}
     \sup \limits_{x \in [0,1]} |p(x)| \leq C \left( T \ln (T) + \sqrt{T N\ln(T)} + \sqrt{T N }\sqrt{u} +  Tu \right).
\end{equation*}
Using this result in conjunction with \eqref{eq : main HDP2} gives us with the same probability
\begin{equation}\label{eq: second main part}
       |L^\top L - (N+1-T)I_T|_{S_\infty} \lesssim  T \ln (T) + \sqrt{T N\ln(T)} + \sqrt{T N}\sqrt{u} +  Tu.
\end{equation}
We will bound the third term in \eqref{eq: 3 terms to upper-bound} the same way we did for the second. Observe that
\begin{align*}
L^\top S=\begin{bmatrix}
0 &\innerl{L_1}{S_2} &\innerl{L_1}{S_3} &\dots &\innerl{L_1}{S_T} \\
\innerl{L_T}{S_2} & 0 &\innerl{L_T}{S_2} &\ddots &\innerl{L_2}{S_T} \\
\innerl{L_T}{S_3} &\innerl{L_1}{S_2} &0 &\ddots &\innerl{L_T}{L_T} \\
\vdots &\ddots &\ddots &\ddots &\vdots\\
\innerl{L_T}{S_T} &\innerl{L_T}{S_{T-1}} &\dots &\innerl{L_T}{S_2} &0
\end{bmatrix}. 
\end{align*}
Define
\begin{align*}
    V = [x_1,\dots,x_T,x_{T-1},\dots,x_0,x_{N-T+2},\dots,x_N,x_{N-2T+1},\dots,x_{N-T+1}],
\end{align*}
and
\begin{align*}
    w(x) =\exp{(2 \pi i x)}.
\end{align*}
By the same reasoning as for the second term in \eqref{eq: 3 terms to upper-bound}
we define the multiplication polynomial over $x \in [0,1]$ by the random quadratic form
\begin{equation*}
    r(x) = \sum \limits_{l=1}^{T}\innerl{L_1}{S_l} w^l(x) + \innerl{L_T}{S_l} w^{-l}(x) = V^\top \mathcal{W}(x)V.
\end{equation*}
Than
\begin{align}\label{eq: random quadratic form}
    |L^\top S|_{S_\infty} = \sup \limits_{x \in [0,1]} |V^\top \mathcal{W}(x)V|. 
\end{align}
 This matrix has only two isolated square lower triangular blocks with non-zero elements 
\begin{align*}
\mathcal{W}_1(x)=\begin{bmatrix}
w(x) &0 &0 &\dots &0 \\
w^2(x) & w(x) &0 &\ddots &0 \\
w^3(x) &w^2(x) & w(x) &\ddots &0 \\
\vdots &\ddots &\ddots &\ddots &\vdots\\
w^T(x) &w^{T-1}(x) & w^{T-2}(x) &\dots &w(x)
\end{bmatrix}. 
\end{align*}
and
\begin{align*}
\mathcal{W}_2(x)=\begin{bmatrix}
w^{-1}(x) &0 &0 &\dots &0 \\
w^{-2}(x) & w^{-1}(x) &0 &\ddots &0 \\
w^{-3}(x) &w^{-2}(x) & w^{-1}(x) &\ddots &0 \\
\vdots &\ddots &\ddots &\ddots &\vdots\\
w^{-T}(x) &w^{-(T-1)}(x) & w^{-(T-2)}(x) &\dots &w^{-1}(x)
\end{bmatrix}. 
\end{align*}
We observe that
\begin{align}
    d_{S_\infty}  (\mathcal{W}_1(x),\mathcal{W}_1(y)) &= \left( \sum \limits_{l=1}^{T} |w^k(x)-w^k(y)|^2_2 \right)^{1/2} \vee \left( \sum \limits_{l=1}^{T} |w^{-k}(x)-w^{-k}(y)|^2_2 \right)^{1/2} \nonumber\\
    &= \left( \sum \limits_{l=1}^{T} 2(1 - \cos(2 \pi k (x-y))) \right)^{1/2} \nonumber\\
    &= 2 \left( \sum \limits_{l=1}^{T} \sin^2(\pi k (x-y)) \right)^{1/2}. \label{eq: infinity distance estimate}
\end{align}
An estimate for $d_{S_2}  (\mathcal{W}(x),\mathcal{W}(y))$ is given by
\begin{align}
    d_{S_2}  (\mathcal{W}(x),\mathcal{W}(y))= \left(\sum \limits_{l=0}^{T} l (|w^{l}(x) - w^{l}(y)|^2_2 + |w^{-l}(x) - w^{-l}(y)|^2_2)\right)^{1/2}.\label{eq: 2 distance estimate}
\end{align}
As above, we use Hanson-Wright inequality to show that the random quadratic form in \eqref{eq: random quadratic form} defines a mixed tail process with pseudo-metrics associated to the distances in \eqref{eq: infinity distance estimate} and \eqref{eq: 2 distance estimate}, than we use the generic chaining result of \citet[Theorem ~$2.2.23$]{talagrand2006generic} and ~\citet[Theorem ~$3.5$]{dirksen2015} to get the desired bound for $u \geq 1$
\begin{equation}\label{eq: HDP10 on polynomial}
    \mathbb{P}\left( \sup \limits_{x \in [0,1]} |r(x)|\geq C \left( E +\sqrt{u} V + u U \right)\right) \leq \frac{1}{3} \exp{(-u)}.
\end{equation}
Here, 
\begin{align*}
    E &= \gamma_2([0,1],| \mathcal{W}(\cdot) |_{S_2}) + \gamma_1([0,1],| \mathcal{W}(\cdot) |_{S_\infty}), \\
    V &= \Delta_{S_2} (\mathcal{W}^{-1}([0,1])),\quad
    U= \Delta_{S_\infty}  (\mathcal{W}^{-1}([0,1])).
\end{align*}
Similar to \eqref{eq : HDP1} the radius of the set $[0,1]$ becomes
\begin{equation*}
    U= \Delta_{S_\infty}  (\mathcal{H}^{-1}([0,1])) \leq T. 
\end{equation*}
Similar to \eqref{eq : HDP3}, the radius of the set $[0,1]$ becomes
\begin{equation*}
    V= \Delta_{S_2}  (\mathcal{W}^{-1}([0,1])) = \sup \limits_{x \in [0,1]} \left(\sum \limits_{l=0}^{T} l (|w^{l}(x)|^2_2 + |w^{-l}(x)|^2_2)\right)^{1/2} \leq T. 
\end{equation*}
Since the cosine function is $1$-Lipschitz, we can evaluate the $\gamma_1$ functional same way we did in \eqref{eq : HDP2} to obtain 
\begin{align*}
    \gamma_1([0,1],| \mathcal{H}(\cdot) |_{S_\infty}) \simeq T \ln(T), 
\end{align*}
and the $\gamma_2$ functional similarly to \eqref{eq: HDP4} to obtain
\begin{align*}
    \gamma_2([0,1],| \mathcal{H}(\cdot) |_{S_2}) \lesssim  T \sqrt{\ln(T)}. 
\end{align*}
The last four results  provide all the estimates needed for equation \eqref{eq: HDP10 on polynomial}, so we get with probability at least $1-\frac{1}{3} \exp{(-u)}$ for $u \geq 1$:
\begin{equation*}
     \sup \limits_{x \in [0,1]} |r(x)| \leq C \left( T \ln(T) + T\sqrt{\ln(T)} + T\sqrt{u} +  Tu \right).
\end{equation*}
This implies that with at least the same probability
\begin{equation*}
       |L^\top S|_{S_\infty} \lesssim   T\ln(T) +  Tu.
\end{equation*}
A straight forward union bound using the events of the last display and equations \eqref{eq : main HDP2} and \eqref{eq: second main part} imply that with probability at least $1-e^{-u}$ for all $u \geq 1$ we have:
\begin{equation*}
    |X^\top X - \E (X^\top X) |_{S_\infty} \lesssim \big( T \ln(T) + \sqrt{T N \ln(T)} \\
    + \sqrt{N T}u^{1/2} + Tu \big).
\end{equation*}
The result of the theorem is obtained by dividing by $N$ in the last display and rescaling by $L^2_x$. The bounds on the eigenvalues follow immediately.
\end{proof}

\begin{proof}[Proof of Theorem~\ref{thm: Bound on the multiplier process}]
After re-normalizing the $x_i$'s and $\varepsilon_i$'s by there respective variances, we might assume that $\sigma_x = \sigma_\varepsilon = 1$.  Observe that
\begin{align*}
    \E\left( \left| \frac{1}{N}X^\top \varepsilon \right|_2^2 \right) &= \frac{1}{N^2} \E \left[\sum \limits_{j=1}^T \left(\sum \limits_{i=1}^N \varepsilon_i x_{i+j-1} \right)^2 \right] \\
    &=  \frac{1}{N^2} \sum \limits_{j=1}^T \sum \limits_{i_1,i_2=1}^N \E (\varepsilon_{i_1}\varepsilon_{i_2} x_{i_1+j-1}x_{i_2+j-1}) \\
    &=  \frac{1}{N^2} \sum \limits_{j=1}^T \sum \limits_{i=1}^N \E (\varepsilon_i^2) \E (x_j^2)= \frac{T}{N}.
\end{align*}
We need then to estimate $|| X^\top \varepsilon|_2 -\sqrt{TN}| _{\psi_2}$. For this, we condition on the knowledge of $X$ to deduce that, for unit vectors $w \in \mathbb{R}^N$ and $s \in \mathbb{R}$ we get by sub-Gaussionality for all $i \in \llbracket 1,N\rrbracket$,
\begin{align*}
    \E(\exp(s \varepsilon_i \inner{w}{X_i}) | X) &\leq \exp(cs^2 | \varepsilon_i| _{\psi_2}^2 \inner{w}{X_i}^2) \\
    &\leq \exp(cs^2 L_{\varepsilon}^2 | X_i| ^2).
\end{align*}
Thus, for all covariate vectors $X_i$,
\begin{align*}
    \E(\exp(s | \varepsilon_i X_i|_2 ) | X) 
    &\leq \exp(cs^2 L_{\varepsilon}^2 | X_i| ^2).
\end{align*}
Conditioning again on $X$, the centring and square sum properties of the $\psi_2$ norm for independent sub-Gaussian random variables \cite[Chapter ~$2$]{vershynin_2018} leads to
\begin{align*}
    \left| |X^\top \varepsilon|_2 -\sqrt{TN}\right| _{\psi_2}^2 &\lesssim \left| \sum \limits_{i=1}^T \varepsilon_i X_i \right|_{\psi_2}^2 
    \leq \sum \limits_{i=1}^T |  \varepsilon_i X_i| _{\psi_2}^2  \\ & \lesssim L_{\varepsilon}^2\sum \limits_{i=1}^T | X_i|_2^2 = L_{\varepsilon}^2T\sum \limits_{i=2-T}^{N} x_i^2. 
\end{align*}

For $s \lesssim (\sqrt{TN} L_\varepsilon L_x)^{-1}$, we obtain the following bound on the moment generating function of $| X^\top \varepsilon|_2$:
\begin{align*}
    \E\{\exp[s(| X^\top \varepsilon|_2 -\sqrt{TN})]\} &= \E\{ \E\{\exp[s(| X^\top \varepsilon|_2 -\sqrt{TN})]|X\}\} \\
    &\leq \E(\exp(cs^2 L_\varepsilon^2T\sum_{i=2-T}^N x_i^2) ) \\
    &\leq \E( \exp(cs^2 L_\varepsilon^2TN x_1^2) ) \\
    &\leq \exp(c_1s^4 T^2N^2 L_\varepsilon^4 L_x^4). 
\end{align*} 
This leads to the following probability bound for $s \lesssim (\sqrt{TN} L_\varepsilon L_x)^{-1}$ and all $\rho > 0$:
\begin{align*}
    \mathbb{P} \left( \left| \left|\frac{1}{N}X^\top \varepsilon\right|_2 - \sqrt{\frac{T}{N}} \right| \geq \rho \right) \leq 2\exp(-s N \rho + c_1s^4 T^2N^2 L_\varepsilon^4 L_x^4). 
\end{align*}
Optimizing for the choice of $s$ we obtain that for the value 
$$s \simeq \frac{1}{\sqrt{TN} L_\varepsilon L_x} \wedge \frac{ \rho^\frac{1}{3} }{T^\frac{2}{3} N^\frac{1}{3}  L_\varepsilon^\frac{4}{3}L_x^\frac{4}{3}},$$
with probability at least
\begin{align*}
1-\exp\left[-c\min\left( \left(\frac{N}{T} \right)^\frac{1}{2}\frac{\rho}{L_\varepsilon L_x}, \left(\frac{N}{T}\right)^\frac{2}{3}\frac{\rho^\frac{4}{3}}{  L_\varepsilon^\frac{4}{3}L_x^\frac{4}{3}}\right)\right] 
\end{align*}
it holds that
\begin{equation*}
    \left| \frac{1}{N}X^\top \varepsilon \right|_2 \leq \sqrt{\frac{T}{N}} + \rho,
\end{equation*}
which implies that with probability at least $1-\eta$,
\begin{align*}
    \left| \frac{1}{N}X^\top \varepsilon \right|_2 \leq \sqrt{\frac{T}{N}} \left[1+c L_\varepsilon L_x \left(\ln\frac{1}{\eta} \vee \ln^\frac{3}{4}\frac{1}{\eta} \right) \right]. 
\end{align*}
The conclusion of the Theorem follows after taking out the re-normalization.
\end{proof}

\bibliographystyle{elsarticle-harv} 
\bibliography{bibliography}

\end{document}